\documentclass[11pt,times,twoside]{article}
\usepackage{graphicx,latexsym,euscript,makeidx,color,bm}
\usepackage{amsmath,amsfonts,amssymb,amsthm,thmtools,mathtools,mathrsfs,enumerate}
\usepackage[colorlinks,linkcolor=blue,anchorcolor=green,citecolor=red]{hyperref}

\usepackage{geometry}
\def\5n{\negthinspace \negthinspace \negthinspace \negthinspace \negthinspace }
\def\4n{\negthinspace \negthinspace \negthinspace \negthinspace }
\def\3n{\negthinspace \negthinspace \negthinspace }
\def\2n{\negthinspace \negthinspace }
\def\1n{\negthinspace }

\def\ss{\smallskip}                
\def\ms{\medskip}                
               
\def\ds{\displaystyle}

\def\no{\noindent}        \def\q{\quad}                      
\def\ns{\noalign{\ss}}    \def\qq{\qquad}

         \def\rf{\eqref}                    
  \def\deq{\triangleq}               
            \def\({\Big (}
                  \def\){\Big )}
\def\leq{\leqslant}       \def\geq{\geqslant}
          \def\[{\Big[}
           \def\]{\Big]}
                   \def\cd{\cdot}

            \def\d{\delta}

\def\bde{\begin{definition}\label}    \def\ede{\end{definition}}
\def\be{\begin{equation}}
\def\bel{\begin{equation}\label}      \def\ee{\end{equation}}
\def\bt{\begin{theorem}\label}        \def\et{\end{theorem}}
\def\bc{\begin{corollary}\label}      \def\ec{\end{corollary}}
\def\bl{\begin{lemma}\label}          \def\el{\end{lemma}}
\def\bp{\begin{proposition}\label}    \def\ep{\end{proposition}}
\def\bas{\begin{assumption}\label}    \def\eas{\end{assumption}}
\def\br{\begin{remark}\label}         \def\er{\end{remark}}
\def\bex{\begin{example}\label}       \def\ex{\end{example}}
\def\ba{\begin{array}}                \def\ea{\end{array}}
\def\ben{\begin{enumerate}}           \def\een{\end{enumerate}}

\newtheorem{theorem}{Theorem}[section]
\newtheorem{definition}[theorem]{Definition}
\newtheorem{proposition}[theorem]{Proposition}
\newtheorem{corollary}[theorem]{Corollary}
\newtheorem{lemma}[theorem]{Lemma}
\newtheorem{remark}[theorem]{Remark}
\newtheorem{example}[theorem]{Example}

\newtheorem{assumption}{Assumption}

\makeatletter
   
   \@addtoreset{equation}{section}
\makeatother

\allowdisplaybreaks[4]

\begin{document}

\title{\Large \bf Fractional backward stochastic differential equations with delayed generator}

\author{
%
Jiaqiang Wen\thanks{Department of Mathematics, Southern University of Science and Technology, Shenzhen 518055, China (Email: {\tt wenjq@sustech.edu.cn}). JW is supported by National Natural Science Foundation of China (grant No. 12101291) and Guangdong Basic and Applied Basic Research Foundation (grant No. 2022A1515012017).}}
\date{}
\maketitle

\renewcommand{\thefootnote}{\fnsymbol{footnote}}

\maketitle

\no\bf Abstract: \rm
In this paper, we focus on the solvability of a class of fractional backward stochastic differential equations (BSDEs, for short) with delayed generator. In this class of equations, the generator includes not only the values of the solutions of the present but also the past. Under Lipschitz condition, the existence and uniqueness of such BSDEs are established. A comparison theorem for this class of BSDEs is also obtained.

\ms

\no\bf Key words: \rm  fractional backward stochastic differential equation; backward stochastic differential equation; fractional Brownian motion; time delayed.

\ms

\no\bf AMS subject classifications. \rm 60H10; 60G22.

\section{Introduction}

A centered Gaussian process $B^{H}=\{ B^{H}_{t},t\geq 0 \}$ is called a
fractional Brownian motion (fBm, for short) with Hurst parameter $H \in (0, 1)$ if its covariance is
\begin{equation*}
 E(B^{H}_{t}B^{H}_{s}) = \frac{1}{2} (t^{2H} + s^{2H} - |t-s|^{2H}), \q~ t,s \geq 0.
\end{equation*}
When $H=\frac{1}{2}$, this process degenerates into a classical Brownian motion.
For $H>\frac{1}{2}$, $B^{H}_{t}$ exhibits the property of long-range dependence,
which makes fBm an important driving noise in many fields such as finance, telecommunication networks, and physics.

\ms

In 1990, Pardoux and Peng \cite{Peng} introduced the nonlinear backward stochastic differential equations (BSDEs, for short). In the next two decades, it had been widely used in different fields of mathematical finance \cite{Peng2}, stochastic control \cite{Yong5}, and partial differential equations \cite{Peng92}. At the same time, for better applications, BSDE itself has been developed into many different branches.
For example, recently, BSDEs driven by fractional Brownian motion, also known as fractional BSDEs,
were studied by Hu and Peng \cite{Hu}, where they proved the existence and uniqueness of solutions when the Hurst parameter $H>\frac{1}{2}$. Then Maticiuc and Nie \cite{Maticiuc} obtained some general results of fractional BSDEs through a rigorous approach. Buckdahn and Jing \cite{Buckdahn3} studied fractional mean-field stochastic differential equations (SDEs, for short) with $H>1/2$ and a stochastic control problem. Some other recent developments of fractional BSDEs can be found in Bender \cite{Bender}, Bender and Viitasaari \cite{Bender-Viitasaari 2017}, Borkowska \cite{Borkowska}, Douissi, Wen and Shi
\cite{Douissi-Wen-Shi 2018}, Hu, Jolis and Tindel \cite{Hu6}, Hu, Nualart and Song \cite{Hu5}, Jing \cite{Jing}, Wen and Shi \cite{Wen,Wen-Shi-20}, etc., among theory and applications.
Furthermore, as a natural extension of BSDEs, Delong and Imkeller \cite{Delong,Delong2} studied BSDEs with delayed generator, owing to that mathematical delayed approaches play an important role in many fields, such as stochastic optimal control, financial risks, management of insurance, pricing and hedging (see Delong \cite{Delong3} and the references therein). Shortly after, in application, Chen and Wu \cite{Chen-Wu-10} obtained the maximum principle for stochastic optimal control problem with delay. Along this way, Huang and Shi \cite{Huang-Shi-12} get the Maximum principle for optimal control of fully coupled forward--backward stochastic differential delayed equations.

\ms

As another important development of BSDEs, fractional BSDEs with delayed generator have significant applications in stochastic optimal control problems with delay. However, to our best knowledge, no study about fractional BSDEs with delayed generator is available up to now. Therefore, in this paper, we focus on the solvability of this type of BSDEs. Specifically, we study the following fractional BSDEs with delayed generator:
\begin{equation}\label{0}
  \begin{cases}
    -dY(t)=f(t,\eta(t),Y(t),Z(t),Y(t-\delta),Z(t-\delta)) dt - Z(t) dB_{t}^{H}, \ \ 0\leq t\leq T; \\
     Y(T) = \xi, \ \ Y(t) = \varphi(t), \ \ Z(t)=\psi(t), \ \ \ \ \ \ \ \ \ \ \ \ \ \ \ \ \ \ \ \ \ \ \ \ \ \ \ \ \ \ \ \ \  -\delta\leq t< 0,
  \end{cases}
\end{equation}
where $\delta\in(0,T]$ is a time delay parameter.
We call $\xi(\cd)$ the {\it terminal value} and
$f(\cd)$ the {\it generator} of the corresponding BSDE \rf{0}. And the  {\it delayed generator} means that in \rf{0}, the generator $f(\cd)$ includes not only the values of the solution $(Y,Z)$ of the present but also the past.
It should be pointed out that the stochastic integral in (\ref{0}) is the divergence type integral
(see Decreusefond and \"{U}st\"{u}nel \cite{Decreusefond}, and  Nualart \cite{Nualart}).
In particular, we consider the case of the Hurst parameter $H>\frac{1}{2}$.
First, two different methods are proposed to study the existence and uniqueness of the equation (\ref{0}).
Note that, even in the classical BSDEs with delayed generator, as shown in Delong and Imkeller \cite{Delong},
the existence and uniqueness cannot hold for arbitrary time horizon and time delay.
Similarly, our existence and uniqueness, proved by the first method, hold under a small time horizon.
However, by introducing the second method, the existence and uniqueness hold for arbitrary time horizon.
Besides, for its wide applications to BSDEs, a comparison theorem of such type of equations is obtained.
In the coming future researches, we will focus on the application of this equation in stochastic optimal control.

\ms

This article is organized as follows.
In Section 2, some preliminaries about fractional Brownian motion are presented.
The existence and uniqueness of the equation (\ref{0}) are proved in Section 3,
and a comparison theorem for such fractional BSDEs is derived in Section 4.

\section{Preliminaries}

We recall some basic results about the fBm in this section, which are important for our following study.
For a deeper discussion, the readers may refer to the articles such as Decreusefond and \"{U}st\"{u}nel \cite{Decreusefond}, Hu \cite{Hu3} and  Nualart \cite{Nualart}.

\ms

Denote by $(\Omega,\mathcal{F},P)$ a complete probability space, and under which we assume $B^{H}=\{ B^{H}_{t},t\geq 0 \}$ is a fBm such that the filtration $\mathcal{F}$ is generated by $B^{H}$.
Let $H >\frac{1}{2}$ throughout this paper.
Moreover, we define the function $\phi(x) = H(2H - 1)|x|^{2H-2}$, where $x \in \mathbb{R}$.
Suppose $\xi:[0,T]\rightarrow \mathbb{R}$ and $\eta:[0,T]\rightarrow \mathbb{R}$ are two continuous functions.
Define
\begin{equation}\label{11}
  \langle \xi,\eta \rangle_{T} = \int_0^T \int_0^T \phi(u-v) \xi_{u} \eta_{v} dudv, \ \
  and \ \   \| \xi \|_{T}^{2} =  \langle \xi,\xi  \rangle_{T}.
\end{equation}
It is easy to know that $\langle \xi,\eta \rangle_{T}$ is a Hilbert scalar product.
So, under this scalar product, we can denote by $\mathcal{H}$ the completion of the continuous functions.
Besides, we denote by $\mathcal{P}_{T}$ the set of all polynomials of fBm over the interval $[0,T]$, i.e.,
every element of $\mathcal{P}_{T}$ has the following form
\begin{equation*}
  \Phi(\omega) = h \left(\int_0^T \xi_{1}(t) dB_{t}^{H},...,\int_0^T \xi_{n}(t) dB_{t}^{H} \right),
\end{equation*}
where $h$ is a polynomial function and $\xi_{i}\in\mathcal{H}, i=1,2,...,n$.
In addition,
Malliavin derivative operator $D_{s}^{H}$ of $\Phi\in \mathcal{P}_{T}$ is defined by:
\begin{equation*}
  D_{s}^{H}\Phi = \sum\limits_{i=1}^{n} \frac{\partial h}{\partial x_{i}}
               \left(\int_0^T \xi_{1}(t) dB_{t}^{H},...,\int_0^T \xi_{n}(t) dB_{t}^{H} \right)\xi_{i}(s), \ \ s\in [0,T].
\end{equation*}
Since the derivative operator
  $D^{H}:L^{2}(\Omega,\mathcal{F}, P) \rightarrow (\Omega,\mathcal{F}, \mathcal{H})$ is closable,
one can denote by $\mathbb{D}^{1,2}$ the completion of $\mathcal{P}_{T}$ under the following norm
$$\| \Phi \|^{2}_{1,2} = E|\Phi|^{2} + E\|D^{H}_{s} \Phi\|^{2}_{T}.$$
Furthermore, we introduce the following derivative
\begin{equation}\label{12}
  \mathbb{D}_{t}^{H}\Phi = \int_0^T \phi(t-s) D_{s}^{H}\Phi ds, \ \ t\in[0,T].
\end{equation}
Now, we introduce the adjoint operator of Malliavin derivative operator $D^{H}$.
We call this operator the divergence operator, which represents the divergence type integral and
 is denoted by $\delta(\cdot)$.

\begin{definition}
We call a process $u\in L^{2}(\Omega\times[0,T];\mathcal{H})$  belongs to $Dom(\delta)$,
if there is a random variable $\delta(u)\in L^{2}(\Omega,\mathcal{F},P)$ satisfying the duality relationship:
\begin{equation*}
  E(\Phi\delta(u))=E(\langle D^{H}_{\cdot} \Phi,u \rangle_{T}), \ \ for \ every \  \Phi\in\mathcal{P}_{T}.
\end{equation*}
Moreover, if $u\in Dom(\delta)$, we define $\int_0^T u_{s} d B^{H}_{s}:=\delta(u)$
the divergence type integral of $u$ w.r.t. $B^{H}$.
\end{definition}
It should be pointed out that, in this paper, unless otherwise specified,
 the $d B^{H}$-integral represents the divergence type integral.

\begin{proposition}[Hu \cite{Hu3}, Proposition 6.25]\label{2}
Denote by $\mathbb{L}^{1,2}_{H}$  the space of all processes $F : \Omega\times[0,T] \rightarrow \mathcal{H}$ satisfying
$$E \left( \| F \|_{T}^{2} + \int_0^T \int_0^T |\mathbb{D}_{s}^{H}F_{t}|^{2} dsdt \right) < \infty.$$
Then, if $F \in \mathbb{L}^{1,2}_{H}$, the divergence type integral
$\int_0^T F_{s} dB_{s}^{H}$ exists in $L^{2}(\Omega,\mathcal{F}, P)$, and
$$
  E \left( \int_0^T F_{s} dB_{s}^{H} \right) = 0; \ \
   E \left( \int_0^T F_{s} dB_{s}^{H} \right)^{2}
 =E \left( \| F \|_{T}^{2} + \int_0^T \int_0^T \mathbb{D}_{s}^{H}F_{t} \mathbb{D}_{t}^{H}F_{s} dsdt \right).
$$
\end{proposition}

\begin{proposition}[Hu \cite{Hu3}, Theorem 11.1]\label{10}
 For $i = 1, 2$, let $g_{i}$ and $f_{i}$ be two real valued processes satisfying
 $E \int_0^T (|g_{i}(s)|^{2} + |f_{i}(s)|^{2}) ds < \infty$.
Moreover, assume that $D^{H}_{t}f_{i}(s)$ is continuously differentiable in its arguments
 $(s,t)\in [0,T]^{2}$ for almost every $\omega \in \Omega$, and
$E \int_0^T \int_0^T |\mathbb{D}_{t}^{H}f_{i}(s)|^{2} dsdt < \infty$.
Denote
\begin{equation*}
  X_{i}(t) = \int_0^t g_{i}(s) ds + \int_0^t f_{i}(s) dB_{s}^{H}, \ \ t\in [0,T].
\end{equation*}
Then
\begin{equation*}
\begin{split}
   X_{1}(t)X_{2}(t) =& \int_0^t X_{1}(s)g_{2}(s) ds + \int_0^t X_{1}(s)f_{2}(s) dB_{s}^{H}
                      +\int_0^t X_{2}(s)g_{1}(s) ds \\
                    &+ \int_0^t X_{2}(s)f_{1}(s) dB_{s}^{H} +\int_0^t \mathbb{D}_{s}^{H}X_{1}(s)g_{2}(s) ds + \int_0^t \mathbb{D}_{s}^{H}X_{2}(s)g_{1}(s) ds.
\end{split}
\end{equation*}
\end{proposition}

\section{Existence and uniqueness}

In this section, we study the solvability of the fractional BSDE (\ref{0}).
In particular, we shall propose two different methods to prove the existence and uniqueness of  (\ref{0}).
And at the end of this section, we will make a comparison for these two methods.
In order to do this, let $\eta_{0}$ be a constant, and $b:[0,T]\rightarrow \mathbb{R}$ and $\sigma:[0,T]\rightarrow \mathbb{R}$ be two deterministic differentiable functions with $\sigma_{t} \neq 0$ (then either $\sigma_{t} < 0$ or $\sigma_{t} > 0$). Let
\begin{equation}\label{33}
\eta_{t} = \eta_{0} + \int_0^t b_{s} ds + \int_0^t \sigma_{s} dB_{s}^{H}, \q~ t\in [0,T].
\end{equation}
We recall that (see (\ref{11})),
\begin{equation*}
  \| \sigma \|_{t}^{2} = \int_0^t \int_0^t\phi(u-v)  \sigma_{u} \sigma_{v} dudv
  = H(2H-1) \int_0^t \int_0^t |u-v|^{2H-2} \sigma_{u} \sigma_{v} dudv,
\end{equation*}
therefore, $\frac{d}{dt}(\| \sigma \|_{t}^{2})= 2 \hat{\sigma}_{t} \sigma_{t}>0$ for $t\in (0,T]$, where
$\hat{\sigma}_{t} = \int_0^t \phi(t-v) \sigma_{v} dv$.

\ms

In the following, we study the equation (\ref{0}).
And for the sake of convenience, we first investigate the fractional BSDE with delayed generator as follows:
\begin{equation}\label{1}
  \begin{cases}
    -dY(t)=f(t,\eta(t),Y(t-\delta),Z(t-\delta)) dt - Z(t) dB_{t}^{H}, \q~ 0\leq t\leq T; \\
     Y(T) = \xi, \ \ Y(t) = \varphi(t), \ \ Z(t)=\psi(t), \q~ \ \ \ \ \ \ \ \ \ \ \ \ \ \ \ \ -\delta\leq t< 0,
  \end{cases}
\end{equation}
where the time delay $\delta\in(0,T]$.
We introduce the following sets: for $t_{1}\leq t_{2}$,
\begin{itemize}
  \item [$\bullet$] $C^{1,3}_{pol}([t_{1},t_{2}] \times \mathbb{R})$
        is the space of all $C^{1,3}$-functions over $[t_{1},t_{2}] \times \mathbb{R}$, which together with their derivatives are of polynomial growth;
  \item [$\bullet$] $\mathcal{V}_{[t_{1},t_{2}]}
  := \bigg \{ Y=\phi\big(\cdot,\eta(|\cdot|)\big) \big| \phi \in C_{pol}^{1,3}([t_{1},t_{2}]\times \mathbb{R})
  \  with \ \frac{\partial \phi}{\partial t}  \in
    C_{pol}^{0,1}([t_{1},t_{2}]$ $\times \mathbb{R}), t_{1}\leq t\leq t_{2} \bigg \},$
\end{itemize}
and by $\widetilde{\mathcal{V}}_{[t_{1},t_{2}]}$ and $\widetilde{\mathcal{V}}_{[t_{1},t_{2}]}^{H}$  denote the completion of $\mathcal{V}_{[t_{1},t_{2}]}$ under the following norm respectively,
\begin{equation*}
  \| Y \| \deq \bigg(\int_{t_{1}}^{t_{2}}  e^{\beta t} \mathbb{E} |Y(t)|^{2} dt\bigg)^{\frac{1}{2}}, \ \ \ \
  \| Z \| \deq \bigg(\int_{t_{1}}^{t_{2}} |t|^{2H-1} e^{\beta t} \mathbb{E} |Z(t)|^{2} dt\bigg)^{\frac{1}{2}},
\end{equation*}
where $\beta\geq 0$ is a constant.
Then  $\widetilde{\mathcal{V}}_{[t_{1},t_{2}]}^{H}$ and
$\widetilde{\mathcal{V}}_{[t_{1},t_{2}]}$ are Banach spaces.
A pair $(Y_{\cdot},Z_{\cdot})\in\widetilde{\mathcal{V}}_{[-\delta,T]} \times \widetilde{\mathcal{V}}_{[-\delta,T]}^{H}$
is said to be a solution of the equation (\ref{1}), if it satisfies (\ref{1}).


\subsection{The first method}

In this subsection, the first method is used to prove the existence and uniqueness of the equation (\ref{1}).
In order to do this, the following assumptions are needed.

\begin{itemize}
  \item [(H1)] $\xi=h(\eta(T))$, where $h\in C^{3}_{pol}(\mathbb{R})$ and,
               $\varphi\in \widetilde{\mathcal{V}}_{[-\delta,0]}$ and
               $\psi\in \widetilde{\mathcal{V}}_{[-\delta,0]}^{H}$.
\end{itemize}
\begin{itemize}
  \item [(H2)] The generator $f(t,x,y,z):[0,T]\times \mathbb{R}^{3}\rightarrow \mathbb{R}$ is a $C_{pol}^{0,1}$-continuous function. Moreover, there exists $L\geq 0$ such that $f$ satisfies the following Lipschitz condition:
\begin{equation*}\label{H}
|f(t,x,y,z) - f(t,x,y',z')| \leq L\big(|y-y'| + |z-z'| \big), \ \ \forall t\in [0,T], x,y,y',z,z' \in \mathbb{R}.
\end{equation*}
\end{itemize}
\begin{theorem}\label{30}
Under (H1) and (H2), for a sufficiently small time horizon $T$,
Eq. (\ref{1}) admits a unique solution.
\end{theorem}

\begin{proof}
For a given pair $(y_{\cdot},z_{\cdot}) \in \widetilde{\mathcal{V}}_{[-\delta,T]} \times \widetilde{\mathcal{V}}_{[-\delta,T]}^{H}$,
we consider the following BSDE:
\begin{equation}\label{3}
  \begin{cases}
    -dY_{t}=g(t,\eta_t)dt - Z_{t} dB_{t}^{H}, \qq\qq 0\leq t\leq T; \\
     Y_{t} = \xi, \ \ Y_{t} = \varphi(t), \ \ Z_{t}=\psi(t),  \qq  -\delta\leq t< 0,
  \end{cases}
\end{equation}
where
$$g(t,\eta_t)=f(t,\eta_{t},y_{t-\delta},z_{t-\delta}),\q~0\leq t\leq T.$$
Note that $(y_{\cdot},z_{\cdot})$ and $\d$ are given.
From Proposition 17 of Maticiuc and Nie \cite{Maticiuc},
and note that $Y_{\cdot}$ and $Z_{\cdot}$ are given in $[-\delta,0)$,
we see that Eq. (\ref{3}) has a unique solution $(Y_{\cdot},Z_{\cdot})$.
Then, we can define a mapping
$$I:\widetilde{\mathcal{V}}_{[-\delta,T]} \times \widetilde{\mathcal{V}}_{[-\delta,T]}^{H}\longrightarrow
\widetilde{\mathcal{V}}_{[-\delta,T]} \times \widetilde{\mathcal{V}}_{[-\delta,T]}^{H}$$
such that $I[(y_{\cdot},z_{\cdot})]=(Y_{\cdot},Z_{\cdot})$.
It is easy to know that, if we can prove $I$ is a contraction mapping on $\widetilde{\mathcal{V}}_{[-\delta,T]} \times \widetilde{\mathcal{V}}_{[-\delta,T]}^{H}$, then the desired result obtained.
So, in the following, we are going to show that $I$ is a contraction mapping.

\ms

For arbitrary pairs $(y_{\cdot},z_{\cdot})$ and
 $(y'_{\cdot},z'_{\cdot})$ in $\widetilde{\mathcal{V}}_{[-\delta,T]} \times \widetilde{\mathcal{V}}_{[-\delta,T]}^{H}$,
we let $$I[(y_{\cdot},z_{\cdot})]=(Y_{\cdot},Z_{\cdot}), \ \ \ I[(y'_{\cdot},z'_{\cdot})]=(Y'_{\cdot},Z'_{\cdot}).$$
And define
$$\ba{ll}
\ds \hat{y}_{\cdot}\triangleq y_{\cdot}-y'_{\cdot}, \q~ \hat{z}_{\cdot}\triangleq z_{\cdot}-z'_{\cdot},\\
\ns\ds\hat{Y}_{\cdot}\triangleq Y_{\cdot}-Y'_{\cdot}, \q~ \hat{Z}_{\cdot}\triangleq Z_{\cdot}-Z'_{\cdot}.
\ea$$
Now, by applying It\^{o} formula (Proposition \ref{10}) and taking expectation, we obtain
$$\ba{ll}
\ds \mathbb{E}\left(e^{\beta t}\hat{Y}_{t}^{2} + \beta \int_t^T e^{\beta s}\hat{Y}_{s}^{2} ds
     + 2\int_t^T e^{\beta s}\mathbb{D}_{s}^{H} \hat{Y}_{s}\hat{Z}_{s} ds\right)\\
\ns\ds=2 \mathbb{E}\int_t^T e^{\beta s}
      \hat{Y}_{s}\big[f(s,\eta_{s},y_{s-\delta},z_{s-\delta}) - f(s,\eta_{s},y'_{s-\delta},z'_{s-\delta})\big] ds.
\ea$$
The main difficulty in the above equation comes from the term of the Malliavin derivative $\mathbb{D}_{s}^{H} \hat{Y}_{s}$.
By virtue of the result obtained in  Hu and Peng \cite{Hu}, and Maticiuc and Nie \cite{Maticiuc}, we have the relation
$\mathbb{D}_{s}^{H} \hat{Y}_{s} = \frac{\hat{\sigma}_{s}}{\sigma_{s}} \hat{Z}_{s}$ (see \cite{Hu,Maticiuc} for detail).
Furthermore, from Remark 6 of \cite{Maticiuc}, there is a constant $M>0$ such that
\begin{equation*}
\frac{s^{2H-1}}{M}\leq \frac{\hat{\sigma}_{s}}{\sigma_{s}}\leq M s^{2H-1}, \ \ \forall s\in [0,T].
\end{equation*}
In the following, for the technical reasons, we choose $M>2$. Then, we deduce
$$\ba{ll}
\ds\mathbb{E}\left(e^{\beta t}\hat{Y}_{t}^{2} + \beta \int_t^T e^{\beta s}\hat{Y}_{s}^{2} ds
     + \frac{2}{M}\int_t^T e^{\beta s}s^{2H-1}\hat{Z}_{s}^{2} ds\right)\\
\ns\ds\leq  2 \mathbb{E}\int_t^T e^{\beta s}
      \hat{Y}_{s}\big[f(s,\eta_{s},y_{s-\delta},z_{s-\delta}) - f(s,\eta_{s},y'_{s-\delta},z'_{s-\delta})\big] ds.
\ea$$
So by choosing $\beta>1$, from (H2), we have
\bel{38}\ba{ll}
\ds\mathbb{E}\left(e^{\beta t}|\hat{Y}_{t}|^{2} + \int_t^T e^{\beta s}|\hat{Y}_{s}|^{2} ds
       + \frac{2}{M}\int_t^T e^{\beta s}s^{2H-1}|\hat{Z}_{s}|^{2} ds\right) \\
\ns\ds\leq2L\int_t^Te^{\beta s}\mathbb{E}\left[|\hat{Y}_{s}|(|\hat{y}_{s-\delta}|+|\hat{z}_{s-\delta}|)\right]ds\\
\ns\ds\leq 2L\int_t^T \big(e^{\beta s} \mathbb{E}|\hat{Y}_{s}|^{2}\big)^{\frac{1}{2}}
       \cdot \left[e^{\beta s} \mathbb{E}(|\hat{y}_{s-\delta}| + |\hat{z}_{s-\delta}|)^{2}\right]^{\frac{1}{2}} ds.
\ea\ee
Denote $x(t)=\big(e^{\beta t} \mathbb{E}|\hat{Y}_{t}|^{2}\big)^{\frac{1}{2}}$. Then from (\ref{38}),
\begin{equation}\label{42}
  x(t)^{2}\leq  2L\int_t^T x(s) \left[e^{\beta s} \mathbb{E}(|\hat{y}_{s-\delta}| + |\hat{z}_{s-\delta}|)^{2}\right]^{\frac{1}{2}} ds.
\end{equation}
By applying Lemma 20 of \cite{Maticiuc} to (\ref{42}), it follows that
$$\ba{ll}
\ds x(t)\leq L\int_t^T \left[e^{\beta s} \mathbb{E}(|\hat{y}_{s-\delta}| + |\hat{z}_{s-\delta}|)^{2}\right]^{\frac{1}{2}} ds\\
\ns\ds\qq \leq \sqrt{2}L \int_t^T \left[e^{\beta s} \mathbb{E}(|\hat{y}_{s-\delta}|^{2} + |\hat{z}_{s-\delta}|^{2})\right]^{\frac{1}{2}} ds.
\ea$$
Therefore
\begin{equation}\label{8}
  x(t)^{2}\leq 2L^{2} \bigg(\int_{0}^{T} \left[e^{\beta s} \mathbb{E}(|\hat{y}_{s-\delta}|^{2} + |\hat{z}_{s-\delta}|^{2})\right]^{\frac{1}{2}} ds\bigg)^{2}, \ \ t\in[0,T].
\end{equation}
From the inequality $\sqrt{|a|+|b|}\leq \sqrt{|a|}+\sqrt{|b|}$ and H\"{o}lder's inequality, note that $\delta\leq T$, one has
\bel{4}\ba{ll}
\ds\bigg(\int_{0}^T \left[e^{\beta s}
   \mathbb{E}(|\hat{y}_{s-\delta}|^{2} + |\hat{z}_{s-\delta}|^{2})\right]^{\frac{1}{2}} ds\bigg)^{2}\\
\ns\ds\leq e^{\beta \delta} \bigg(\int_{-\delta}^T \left[e^{\beta s} \mathbb{E}(|\hat{y}_{s}|^{2}+|\hat{z}_{s}|^{2})\right]^{\frac{1}{2}} ds\bigg)^{2}\\
\ns\ds\leq e^{\beta \delta} \bigg(\int_{-\delta}^T \left[e^{\beta s} \mathbb{E}|\hat{y}_{s}|^{2}\right]^{\frac{1}{2}} ds+\int_{-\delta}^T \left[e^{\beta s} \mathbb{E}|\hat{z}_{s}|^{2}\right]^{\frac{1}{2}}ds\bigg)^{2} \\
\ns\ds\leq 2e^{\beta \delta}\bigg[\bigg(\int_{-\delta}^T \left[e^{\beta s} \mathbb{E}|\hat{y}_{s}|^{2} \right]^{\frac{1}{2}} ds \bigg)^{2}+\bigg(\int_{-\delta}^T \big[|s|^{1-2H} \cdot e^{\beta s} |s|^{2H-1}\mathbb{E}|\hat{z}_{s}|^{2}\big]^{\frac{1}{2}}ds\bigg)^{2}\bigg] \\
\ns\ds\leq 2e^{\beta \delta}\bigg[(T+\delta)\int_{-\delta}^Te^{\beta s} \mathbb{E}|\hat{y}_{s}|^{2} ds
+\frac{T^{2-2H}+\delta^{2-2H}}{2-2H} \int_{-\delta}^Te^{\beta s} |s|^{2H-1}\mathbb{E}|\hat{z}_{s}|^{2} ds\bigg]\\
\ns\ds\leq 2e^{\beta T}\bigg(2T+\frac{2T^{2-2H}}{2-2H}\bigg)
\int_{-\delta}^Te^{\beta s} \mathbb{E} \left(|\hat{y}_{s}|^{2} + |s|^{2H-1}|\hat{z}_{s}|^{2}\right) ds.
\ea\ee
Then combining (\ref{8}) and  (\ref{4}), one has
\begin{equation}\label{9}
 \int_{0}^T x(s)^{2} ds \leq 8L^{2}Te^{\beta T}\bigg(T+\frac{T^{2-2H}}{2-2H}\bigg)
\int_{-\delta}^Te^{\beta s} \mathbb{E} \left(|\hat{y}_{s}|^{2} + |s|^{2H-1}|\hat{z}_{s}|^{2}\right) ds.
\end{equation}
Again, from (\ref{8}) and  (\ref{4}), similar as the above discussion, we obtain
\bel{5}\ba{ll}
\ds\int_{0}^T |s-\delta|^{1-2H} x(s)^{2} ds \\
\ns\ds\leq 2L^{2}\int_{0}^T |s-\delta|^{1-2H} ds \cdot\bigg(\int_{0}^T \left[e^{\beta s} \mathbb{E}(|\hat{y}_{s-\delta}|^{2} + |\hat{z}_{s-\delta}|^{2})\right]^{\frac{1}{2}} ds\bigg)^{2}\\
\ns\ds \leq 2L^{2}\int_{-\delta}^T |s|^{1-2H} ds \cdot\bigg(\int_{0}^T \left[e^{\beta s} \mathbb{E}(|\hat{y}_{s-\delta}|^{2} + |\hat{z}_{s-\delta}|^{2})\right]^{\frac{1}{2}} ds\bigg)^{2} \\
\ns\ds\leq 8L^{2}e^{\beta T}\bigg(T+\frac{T^{2-2H}}{2-2H}\bigg) \cdot \frac{2T^{2-2H}}{2-2H}
       \int_{-\delta}^Te^{\beta s} \mathbb{E} \left(|\hat{y}_{s}|^{2} + |s|^{2H-1}|\hat{z}_{s}|^{2}\right) ds.
\ea\ee
Now, by combining (\ref{38}), (\ref{9}) and (\ref{5}), one has
$$\ba{ll}
\ds\mathbb{E}\left( \int_{0}^T e^{\beta s}|\hat{Y}_{s}|^{2} ds
       + \frac{2}{M}\int_{0}^T e^{\beta s}s^{2H-1}|\hat{Z}_{s}|^{2} ds\right)\\
\ns\ds\leq  2L \mathbb{E}\int_{0}^T e^{\beta s} \bigg(\frac{1}{v}\big(1+|s-\delta|^{1-2H} \big)|\hat{Y}_{s}|^{2}
        + v|\hat{y}_{s-\delta}|^{2} + v|s-\delta|^{2H-1}|\hat{z}_{s-\delta}|^{2}  \bigg) ds \\
\ns\ds\leq  \frac{2L}{v}\mathbb{E}\int_{0}^T e^{\beta s}\big(1+|s-\delta|^{1-2H}\big)|\hat{Y}_{s}|^{2} ds
  + 2 Lv e^{\beta \delta} \mathbb{E}\int_{-\delta}^T e^{\beta s}\left(|\hat{y}_{s}|^{2} + |s|^{2H-1}|\hat{z}_{s}|^{2}\right) ds\\
\ns\ds\leq \widetilde{L}
     \cdot\mathbb{E}\int_{-\delta}^T e^{\beta s}\left(|\hat{y}_{s}|^{2} + |s|^{2H-1}|\hat{z}_{s}|^{2}\right) ds,
\ea$$
where $v>0$ is a constant, and
\begin{equation*}
  \widetilde{L}= \frac{16 L^{3}}{v}e^{\beta T}\bigg(T+\frac{T^{2-2H}}{2-2H}\bigg)
\cdot \bigg(T+ \frac{2T^{2-2H}}{2-2H}\bigg) +  2L v e^{\beta T}.
\end{equation*}
Note that $M>2$, and $\hat{Y}_{s}=0$ and  $\hat{Z}_{s}=0$ when $s\in[-\delta,0)$,
we obtain
$$\ba{ll}
\ds\mathbb{E} \int_{-\delta}^T e^{\beta s}\left(|\hat{Y}_{s}|^{2} + |s|^{2H-1}|\hat{Z}_{s}|^{2}\right) ds\\
\ns\ds\leq \bigg[\frac{8 L^{3}}{v}M e^{\beta T}\bigg(T+\frac{2T^{2-2H}}{2-2H}\bigg)^{2} +  LM v e^{\beta T} \bigg]\cdot
     \mathbb{E}\int_{-\delta}^T e^{\beta s}\left(|\hat{y}_{s}|^{2} + |s|^{2H-1}|\hat{z}_{s}|^{2}\right) ds.
\ea$$
Choose $v$ such that $ LM v e^{\beta T} <\frac{1}{4}$,
and $T$ sufficiently small such that
\begin{equation*}
\frac{8 L^{3}}{v}M e^{\beta T}\bigg(T+ \frac{2T^{2-2H}}{2-2H}\bigg)^{2} <\frac{1}{4}.
\end{equation*}
Then
$$
     \mathbb{E} \int_{-\delta}^T e^{\beta s}\left(|\hat{Y}_{s}|^{2} + |s|^{2H-1}|\hat{Z}_{s}|^{2}\right) ds
\leq \frac{1}{2} \mathbb{E}\int_{-\delta}^T e^{\beta s} \left(|\hat{y}_{s}|^{2} + |s|^{2H-1}|\hat{z}_{s}|^{2}\right) ds.
$$
Hence $I$ is a contraction mapping on
$\widetilde{\mathcal{V}}_{[-\delta,T]} \times \widetilde{\mathcal{V}}^{H}_{[-\delta,T]}$,
which implies (\ref{1}) admits a unique solution.
This completes the proof.
\end{proof}

\begin{remark}
One may note that, in Theorem \ref{30}, the time horizon $T$ needs to be sufficiently small.
Next, we introduce the second method to study (\ref{1}),
where the existence and uniqueness of (\ref{1}) hold for arbitrary time horizon $T$.
\end{remark}

\subsection{The second method}

In this subsection, we introduce another method to prove the solvability of BSDE (\ref{1}).
It should be pointed out that this method is more convenient than the first one.
However, the price of doing this is that we should strengthen the condition of the coefficient $f$ w.r.t. $z$.

\begin{itemize}
  \item [(H3)] The generator $f(t,x,y,z):[0,T]\times \mathbb{R}^{3}\rightarrow \mathbb{R}$ is a $C_{pol}^{0,1}$-continuous function. Moreover, there exists $L\geq 0$ such that $f$ satisfies the following condition:
\bel{14}\ba{ll}
\ds|f(t,x,y,z) - f(t,x,y',z')|^{2} \leq L\big(|y-y'|^{2} + |t-\delta|^{2H-1}|z-z'|^{2} \big), \\
\ns\ds\qq\qq\qq \forall t\in [0,T], x,y,y',z,z' \in \mathbb{R}.
\ea\ee
\end{itemize}

We present the following result, which will be used frequently in the following study.
For the detailed proof of the following lemma, the readers may refer to Lemma 3.1 of Wen and Shi \cite{Wen}.

\begin{lemma}\label{35}
Suppose $h$ is a $C^{1}_{pol}(\mathbb{R})$-function and $f$ is a $C_{pol}^{0,1}([0,T]\times \mathbb{R})$-function.
Then BSDE
\begin{equation*}
    Y_{t}=h(\eta_{T}) + \int_t^T f(s,\eta_{s}) ds - \int_t^T Z_{s} dB_{s}^{H},
\end{equation*}
has a unique solution in $\widetilde{\mathcal{V}}_{[0,T]} \times \widetilde{\mathcal{V}}^{H}_{[0,T]}$.
Moreover,
\bel{34}\ba{ll}
\\ds\mathbb{E}\left(e^{\beta t}|Y_{t}|^{2} + \frac{\beta}{2}\int_t^T e^{\beta s}|Y_{s}|^{2} ds
    + \frac{2}{M}\int_t^Te^{\beta s}s^{2H-1}|Z_{s}|^{2} ds\right)\\
\ns\ds\leq \mathbb{E}\left(e^{\beta T}|h(\eta_{T})|^{2}
     + \frac{2}{\beta}\int_t^T e^{\beta s}|f(s,\eta_{s})|^{2} ds \right),
\ea\ee
where $M, \beta> 0$ are constants.
\end{lemma}

\begin{theorem}\label{6}
Under (H1) and (H3), for a small time delay $\delta$, (\ref{1}) has a unique solution.
\end{theorem}

\begin{proof}
First, similar to the preceding method, for a given pair
$(y_{\cdot},z_{\cdot}) \in \widetilde{\mathcal{V}}_{[-\delta,T]} \times \widetilde{\mathcal{V}}_{[-\delta,T]}^{H}$,
we consider the following BSDE:
\begin{equation}\label{7}
  \begin{cases}
    -dY_{t}=g(t,\eta_{t}) dt - Z_{t} dB_{t}^{H}, \qq\qq 0\leq t\leq T; \\
     Y_{t} = \xi, \ \ Y_{t} = \varphi(t), \ \ Z_{t}=\psi(t),  \qq -\delta\leq t< 0,
  \end{cases}
\end{equation}
where
$$g(t,\eta_t)=f(t,\eta_{t},y_{t-\delta},z_{t-\delta}),\q~0\leq t\leq T.$$
Note that $(y_{\cdot},z_{\cdot})$ and $\d$ are given. From Lemma \ref{35}, wee see that
BSDE (\ref{7}) has a unique solution $(Y_{\cdot},Z_{\cdot})$.
Then, we can define a mapping
$$I:\widetilde{\mathcal{V}}_{[-\delta,T]} \times \widetilde{\mathcal{V}}_{[-\delta,T]}^{H}\longrightarrow
\widetilde{\mathcal{V}}_{[-\delta,T]} \times \widetilde{\mathcal{V}}_{[-\delta,T]}^{H}$$
such that $I[(y_{\cdot},z_{\cdot})]=(Y_{\cdot},Z_{\cdot})$.
In the following, we use the second method to show that $I$ is a contraction mapping on $\widetilde{\mathcal{V}}_{[-\delta,T]} \times \widetilde{\mathcal{V}}_{[-\delta,T]}^{H}$.

For arbitrary two pairs $(y_{\cdot},z_{\cdot})$ and
 $(y'_{\cdot},z'_{\cdot})$ in $\widetilde{\mathcal{V}}_{[-\delta,T]} \times \widetilde{\mathcal{V}}_{[-\delta,T]}^{H}$,
we let $$I[(y_{\cdot},z_{\cdot})]=(Y_{\cdot},Z_{\cdot}), \ \ \ I[(y'_{\cdot},z'_{\cdot})]=(Y'_{\cdot},Z'_{\cdot}).$$
And define
$$
\hat{y}_{\cdot}\deq y_{\cdot}-y'_{\cdot}, \q~  \hat{z}_{\cdot}\deq z_{\cdot}-z'_{\cdot},\q~
\hat{Y}_{\cdot}\deq Y_{\cdot}-Y'_{\cdot}, \q~  \hat{Z}_{\cdot}\deq Z_{\cdot}-Z'_{\cdot}.
$$
From the estimate (\ref{34}), we have
$$\ba{ll}
\ds\mathbb{E}\int_0^T e^{\beta s}\left(\frac{\beta}{2}|\hat{Y}_{s}|^{2}  + \frac{2}{M}s^{2H-1}|\hat{Z}_{s}|^{2} \right)ds\\
\ns\ds\leq \frac{2}{\beta} \mathbb{E}\int_0^T e^{\beta s}
      \big|f(s,\eta_{s},y_{s-\delta},z_{s-\delta}) - f(s,\eta_{s},y'_{s-\delta},z'_{s-\delta})\big|^{2} ds.
\ea$$
Then, from (H3) and Fubini's Theorem, we obtain
$$\ba{ll}
\ds\mathbb{E}\int_0^T e^{\beta s}\left(\frac{\beta}{2}|\hat{Y}_{s}|^{2}  + \frac{2}{M}s^{2H-1}|\hat{Z}_{s}|^{2} \right)ds\\
\ns\ds\leq \frac{2L}{\beta} \mathbb{E}\int_0^T e^{\beta s}\left(|\hat{y}_{s-\delta}|^{2} + |s-\delta|^{2H-1}|\hat{z}_{s-\delta}|^{2} \right) ds\\
\ns\ds\leq \frac{2L e^{\beta \delta}}{\beta}
\mathbb{E}\int_{-\delta}^{T} e^{\beta s} \left(|\hat{y}_{s}|^{2} + |s|^{2H-1}|\hat{z}_{s}|^{2} \right) ds.
\ea$$
Or
$$
     \mathbb{E}\int_0^T e^{\beta s} \bigg(\frac{M \beta}{4}|\hat{Y}_{s}|^{2} + s^{2H-1}|\hat{Z}_{s}|^{2} \bigg)ds
\leq \frac{LMe^{\beta \delta}}{\beta}
\mathbb{E}\int_{-\delta}^{T} e^{\beta s} \left(|\hat{y}_{s}|^{2} + |s|^{2H-1}|\hat{z}_{s}|^{2} \right) ds.
$$
Therefore, by choosing $\delta=\frac{1}{\beta}$ with $\beta=2LMe+\frac{4}{M}$, we have
$$
\mathbb{E} \int_{-\delta}^T e^{\beta s}\bigg(|\hat{Y}_{s}|^{2} + |s|^{2H-1}|\hat{Z}_{s}|^{2}\bigg) ds
\leq \frac{1}{2} \mathbb{E}\int_{-\delta}^T e^{\beta s}\left(|\hat{y}_{s}|^{2} + |s|^{2H-1}|\hat{z}_{s}|^{2}\right) ds.
$$
Hence $I$ is a contraction mapping on
$\widetilde{\mathcal{V}}_{[-\delta,T]} \times \widetilde{\mathcal{V}}^{H}_{[-\delta,T]}$,
which implies (\ref{1}) admits a unique solution.
\end{proof}

\begin{remark}
We make a comparison for the above two methods.
It is easy to see that (H2) is weaker than (H3).
So from the point of view of the condition, the first method is better than the second one.
On the other hand, thanks to the concise proof and the arbitrary horizon $T$, the second method is better.
\end{remark}

Now, we return to the general equation (\ref{0}). And the following assumption is needed.

\begin{itemize}
  \item [(H4)] The generator $f(t,x,y,z,y_{\delta},z_{\delta}):[0,T]\times \mathbb{R}^{5}\rightarrow \mathbb{R}$ is
               a $C_{pol}^{0,1}$-continuous function.
               Moreover, there exists $L\geq 0$ such that $f$ satisfies the following condition:
               for every $t\in [0,T], x,y,y',z,z',$ $y_{\delta},y'_{\delta},z_{\delta},z'_{\delta} \in \mathbb{R},$
\bel{15}\ba{ll}
\ds|f(t,x,y,z,y_{\delta},z_{\delta}) - f(t,x,y',z',y'_{\delta},z'_{\delta})|^{2} \\
\ns\ds\leq L\big(|y-y'|^{2} + t^{2H-1}|z-z'|^{2} + |y_{\delta}-y_{\delta}'|^{2} + |t-\delta|^{2H-1}|z_{\delta}-z'_{\delta}|^{2} \big),
\ea\ee
\end{itemize}

We have the following existence and uniqueness  result for the general BSDE (\ref{0}).

\begin{theorem}\label{16}
Under (H1) and (H4), for a small time delay $\delta$,
BSDE (\ref{0}) admits a unique solution
$(Y_{\cdot},Z_{\cdot}) \in \widetilde{\mathcal{V}}_{[-\delta,T]} \times \widetilde{\mathcal{V}}_{[-\delta,T]}^{H}$.
\end{theorem}

\begin{remark}
Since the proof of Theorem \ref{16} is essentially the same as the above second method,
we only present the result without a detailed proof.
Furthermore, similar to the first method in the preceding subsection,
under (H1) and the related Lipschitz condition (H2), if the time horizon $T$ is sufficiently small,
 BSDE (\ref{0}) also has a unique solution.
\end{remark}

\section{Comparison theorem}

For its wide applications to BSDEs, a comparison theorem of the fractional BSDEs with delayed generator is investigated in this section. In detail, we study a comparison theorem for the solutions of the following type of fractional BSDEs with delayed generator:
\begin{equation}\label{41}
  \begin{cases}
    -dY(t)=f(t,\eta(t),Y(t),Z(t),Y(t-\delta)) dt - Z(t) dB_{t}^{H}, \qq 0\leq t\leq T; \\
     Y(T) = h(\eta_{T}), \qq Y(t) = \varphi(t),  \qq  -\delta\leq t< 0.
  \end{cases}
\end{equation}
From Theorem \ref{16}, under (H1) and (H4), for a small time delay $\delta$, the above equation has a unique solution in
$\widetilde{\mathcal{V}}_{[-\delta,T]} \times \widetilde{\mathcal{V}}_{[0,T]}^{H}$.

\begin{theorem}\label{40}
For $i=1,2$, suppose $h_{i}$ and $\varphi_{i}$ satisfy (H1),
$f_{i}(t,x,y,z,y_{\delta})$ and $\partial_{y}f_{i}(t,x,y,z,y_{\delta})$ satisfy (H4)
for every $(t,x,y,z,y_{\delta})\in[0,T]\times \mathbb{R}^{4}$.
Moreover, assume $f_{1}$ is increasing in $y_{\delta}$, i.e.,
$f_{1}(t,x,y,z,y_{\delta})\leq f_{1}(t,x,y,z,y_{\delta}')$ when $y_{\delta}\leq y_{\delta}'$.
Then, if $\psi_{1}(t)\leq\psi_{2}(t)$, $t\in[-\delta,0]$, and
\begin{equation*}
 h_{1}(x)\leq h_{2}(x),  \ \ f_{1}(t,x,y,z,y_{\delta}) \leq f_{2}(t,x,y,z,y_{\delta}),
 \ \ (t,x,y,z,y_{\delta}) \in[0,T] \times \mathbb{R}^{4},
\end{equation*}
one has
\begin{equation*}
Y_{1}(t)\leq Y_{2}(t), \ a.e., \ a.s.
\end{equation*}
\end{theorem}

\begin{proof}
Let $\widetilde{Y}_{0}(\cdot)=Y_{2}(\cdot)$ and consider the following BSDE:
\begin{equation*}
  \begin{cases}
      -d\widetilde{Y}_{1}(t)=
      f_{1}(t,\eta_{t},\widetilde{Y}_{1}(t),\widetilde{Z}_{1}(t),\widetilde{Y}_{0}(t-\delta)) dt
       - \widetilde{Z}_{1}(t) dB_{t}^{H}, \q  0\leq t\leq T; \\
      \widetilde{Y}_{1}(T) = h_{1}(\eta_{T}), \q
      \widetilde{Y}_{1}(t) = \varphi_{1}(t),  \q  -\delta\leq t< 0.
  \end{cases}
\end{equation*}
From Theorem \ref{16}, the above equation has a unique solution
$(\widetilde{Y}_{1}(\cdot),\widetilde{Z}_{1}(\cdot))
\in \widetilde{\mathcal{V}}_{[-\delta,T]} \times \widetilde{\mathcal{V}}_{[0,T]}^{H}$.
Since
\begin{equation*}
  \begin{cases}
   f_{1}(t,x,y,z,\widetilde{Y}_{0}(t-\delta))
   \leq f_{2}(t,x,y,z,\widetilde{Y}_{0}(t-\delta)), \ \ (t,x,y,z)\in [0,T]\times \mathbb{R}^{3};\\
   h_{1}(x)\leq h_{2}(x), \ \ x\in \mathbb{R},
  \end{cases}
\end{equation*}
from Theorem 12.3 of Hu et al. \cite{Hu1}, we have
\begin{equation*}
  \widetilde{Y}_{1}(t)\leq \widetilde{Y}_{0}(t)=Y_{2}(t), \  \ a.e., \ a.s.
\end{equation*}
Next, we consider the following BSDE:
\begin{equation*}
  \begin{cases}
      -d\widetilde{Y}_{2}(t)=
      f_{1}(t,\eta_{t},\widetilde{Y}_{2}(t),\widetilde{Z}_{2}(t),\widetilde{Y}_{1}(t-\delta)) dt
       - \widetilde{Z}_{2}(t) dB_{t}^{H}, \q  0\leq t\leq T; \\
      \widetilde{Y}_{2}(T) = h_{1}(\eta_{T}), \q
      \widetilde{Y}_{2}(t) = \varphi_{1}(t),  \q  -\delta\leq t< 0.
  \end{cases}
\end{equation*}
and denote by
$(\widetilde{Y}_{2}(\cdot),\widetilde{Z}_{2}(\cdot)) \in
\widetilde{\mathcal{V}}_{[-\delta,T]} \times \widetilde{\mathcal{V}}_{[0,T]}^{H}$
the unique solution of the above equation.
Then, since $f_{1}(t,x,y,z,\cdot)$ is increasing,
we have for all $(t,x,y,z)\in [0,T]\times \mathbb{R}^{3}$,
\begin{equation*}
   f_{1}(t,x,y,z,\widetilde{Y}_{1}(t-\delta))
   \leq f_{1}(t,x,y,z,\widetilde{Y}_{0}(t-\delta)).
\end{equation*}
Hence, similar as the above discission,
\begin{equation*}
  \widetilde{Y}_{2}(t)\leq \widetilde{Y}_{1}(t), \ \ a.e., \ a.s.
\end{equation*}
Then, by induction, one can construct a sequence
$\{(\widetilde{Y}_{n}(\cdot),\widetilde{Z}_{n}(\cdot))\}_{n\geq 1} \subseteq \widetilde{\mathcal{V}}_{[-\delta,T]} \times \widetilde{\mathcal{V}}_{[0,T]}^{H}$
such that
\begin{equation*}
  \begin{cases}
      -d\widetilde{Y}_{n}(t)=
      f_{1}(t,\eta_{t},\widetilde{Y}_{n}(t),\widetilde{Z}_{n}(t),\widetilde{Y}_{n-1}(t-\delta)) dt
       - \widetilde{Z}_{n}(t) dB_{t}^{H}, \q  0\leq t\leq T; \\
      \widetilde{Y}_{n}(T) = h_{1}(\eta_{T}), \q
      \widetilde{Y}_{n}(t) = \varphi_{1}(t),  \q  -\delta\leq t< 0.
  \end{cases}
\end{equation*}
Similarly, we obtain
\begin{equation*}
  Y_{2}(t)= \widetilde{Y}_{0}(t)\geq \widetilde{Y}_{1}(t)\geq \widetilde{Y}_{2}(t)\geq
   \cdots \geq \widetilde{Y}_{n}(t)\geq \cdots, \ a.e., \ a.s.
\end{equation*}
In the following, we shall show $\{(\widetilde{Y}_{n}(\cdot),\widetilde{Z}_{n}(\cdot))\}_{n\geq 1}$ is a Cauchy sequence.

Denote
$\hat{Y}_{n}(t)=\widetilde{Y}_{n}(t)-\widetilde{Y}_{n-1}(t)$ and
$\hat{Z}_{n}(t)=\widetilde{Z}_{n}(t)-\widetilde{Z}_{n-1}(t), \ n\geq 4$.
From (\ref{34}) and the assumption (H4), we have
\begin{equation*}
\begin{split}
    & \mathbb{E}\left(\frac{\beta}{2}\int_0^T e^{\beta s}|\hat{Y}_{n}(s)|^{2} ds
    +\frac{2}{M}\int_0^T s^{2H-1}e^{\beta s}|\hat{Z}_{n}(s)|^{2} ds\right)\\
\leq& \frac{2}{\beta}\mathbb{E}\bigg(\int_0^T e^{\beta s}
      \big| f_{1}(s,\eta_{s},\widetilde{Y}_{n}(s),\widetilde{Z}_{n}(s),\widetilde{Y}_{n-1}(s-\delta))\\
    & \ \ \ \ \ \ \ \ \ \ \ \ \ \  -f_{1}(s,\eta_{s},\widetilde{Y}_{n-1}(s),\widetilde{Z}_{n-1}(s),\widetilde{Y}_{n-2}(s-\delta)) \big|^{2} ds\bigg)\\
\leq& \frac{2L}{\beta}\mathbb{E} \int_0^T e^{\beta s}\big(|\hat{Y}_{n}(s)|^{2}+s^{2H-1}|\hat{Z}_{n}(s)|^{2}\big) ds
     +\frac{2L e^{\beta \delta}}{\beta}\mathbb{E} \int_{-\delta}^T e^{\beta s}|\hat{Y}_{n-1}(s)|^{2} ds.
\end{split}
\end{equation*}
Then, by choosing $\delta=\frac{1}{\beta}$ with $\beta=8LMe+\frac{4}{M}$, one has
$$\ba{ll}
\ds \mathbb{E}\int_0^T e^{\beta s}\big(|\hat{Y}_{n}(s)|^{2} + s^{2H-1}|\hat{Z}_{n}(s)|^{2}\big) ds \\
\ns\ds\leq \frac{1}{4}\mathbb{E} \int_0^T e^{\beta s}\big(|\hat{Y}_{n}(s)|^{2}+s^{2H-1}|\hat{Z}_{n}(s)|^{2}\big) ds
     +\frac{1}{4}\mathbb{E} \int_{-\delta}^T e^{\beta s}|\hat{Y}_{n-1}(s)|^{2} ds.
\ea$$
Hence
$$\ba{ll}
\ds\mathbb{E}\int_0^T e^{\beta s}\big(|\hat{Y}_{n}(s)|^{2} + s^{2H-1}|\hat{Z}_{n}(s)|^{2}\big) ds \\
\ns\ds\leq  \frac{1}{3}\mathbb{E}\int_{-\delta}^T e^{\beta s}|\hat{Y}_{n-1}(s)|^{2} ds\\
\ns\ds\leq  \frac{1}{3}\bigg(\mathbb{E}\int_{-\delta}^T e^{\beta s}|\hat{Y}_{n-1}(s)|^{2} ds
          +       \mathbb{E}\int_{0}^T e^{\beta s}s^{2H-1}|\hat{Z}_{n-1}(s)|^{2} ds\bigg).
\ea$$
Therefore
$$\ba{ll}
\ds  \mathbb{E}\int_{-\delta}^T e^{\beta s}|\hat{Y}_{n}(s)|^{2} ds
    + \mathbb{E}\int_{0}^T e^{\beta s}s^{2H-1}|\hat{Z}_{n}(s)|^{2} ds\\
\ns\ds\leq (\frac{1}{3})^{n-4}\bigg(\mathbb{E}\int_{-\delta}^T e^{\beta s}|\hat{Y}_{4}(s)|^{2} ds
          +       \mathbb{E}\int_{0}^T e^{\beta s}s^{2H-1}|\hat{Z}_{4}(s)|^{2} ds\bigg).
\ea$$
It follows that $(\hat{Y}_{n}(\cdot))_{n\geq 4}$ is a Cauchy sequence in Banach space $\widetilde{\mathcal{V}}_{[-\delta,T]}$,
and $(\hat{Z}_{n}(\cdot))_{n\geq 4}$ is a Cauchy sequence in Banach space $\widetilde{\mathcal{V}}_{[0,T]}^{H}$.
Denote their limits by $\widetilde{Y}_{\cdot}$ and $\widetilde{Z}_{\cdot}$, respectively.
Now from the existence and uniqueness theorem (Theorem \ref{16}), we obtain
\begin{equation*}
  \widetilde{Y}(t)= Y_{1}(t), \ \ a.e., \ a.s.
\end{equation*}
Then, we get
\begin{equation*}
Y_{1}(t)\leq Y_{2}(t), \ \ a.e., \ a.s.
\end{equation*}
Therefore, the desired result is obtained.
\end{proof}

\begin{remark}
It should be pointed out that the results are based on the Hurst parameter $H$ greater then $1/2$. And due to the technical difficulty, the theory of the fractional BSDEs with $H < 1/2$ is still an open problem now.
We hope to give some related results when the Hurst parameter $H<\frac{1}{2}$ in the near future.
\end{remark}

\begin{example}
 Suppose we are facing with the following two BSDEs,
\begin{equation*}
  \begin{cases}
    -dY_{1}(t)=\big[Y_{1}(t) + t^{2H-1}Z_{1}(t) + Y_{1}(t-\delta) - 1 \big] dt - Z_{1}(t) dB_{t}^{H}, \q 0\leq t\leq T; \\
     Y_{1}(T) = h_{1}(\eta_{T}), \q Y_{1}(t) = \varphi_{1}(t),  \q  -\delta\leq t< 0,
  \end{cases}
\end{equation*}
and
\begin{equation*}
  \begin{cases}
    -dY_{2}(t)=\big[Y_{2}(t) + t^{2H-1}Z_{2}(t) + Y_{2}(t-\delta) + 1 \big] dt - Z_{2}(t) dB_{t}^{H}, \q 0\leq t\leq T; \\
     Y_{2}(T) = h_{2}(\eta_{T}), \q Y_{2}(t) = \varphi_{2}(t),  \q  -\delta\leq t< 0,
  \end{cases}
\end{equation*}
where for $i=1,2$, $h_{i}$ and $\varphi_{i}$ satisfy (H1) with $h_{1}(x)\leq h_{2}(x)$ and
 $\varphi_{1}(t)\leq \varphi_{2}(t)$ for every $t\in[-\delta,0], \ x\in \mathbb{R}$.
 Then, according to Theorem \ref{40}, we get $ Y_{1}(t)\leq Y_{2}(t), \ \ a.e., \ a.s. $
\end{example}

\section*{Conclusion}

In this paper, we studied the fractional BSDEs with delayed generator. In particular, we consider the case of the Hurst parameter $H>\frac{1}{2}$.
We proposed two different methods to prove the existence and uniqueness of such BSDEs.
We also prove a comparison theorem for this type of equations.
It should be pointed out that the results obtained in this article extend part of
the main results of Delong and Imkeller \cite{Delong,Delong2} to fractional calculus.
In the coming future researches,
we would like to focus on the application of this equation in finance.
The theory of the case when the Hurst parameter $H < \frac{1}{2}$ is anther goal.

\end{document}